\documentclass[12pt]{amsart}
\usepackage{color}
\usepackage{xypic}
\usepackage[colorlinks=true,urlcolor=blue,linkcolor=blue,citecolor=blue]{hyperref}
\usepackage{ amssymb }
\usepackage{amsthm}
\usepackage{thmtools}

\numberwithin{equation}{section}
\theoremstyle{plain}
\newtheorem{theorem}{Theorem}[section]

\newtheorem{lemma}[theorem]{Lemma}
\newtheorem{cor}[theorem]{Corollary}

\theoremstyle{definition}
\newtheorem{conjecture}{Conjecture}
\newtheorem{definition}[theorem]{Definition}
\newtheorem{exa}[theorem]{Example}

\newtheorem{que}[theorem]{Question}

\newcommand{\st}{\ :\ }

\newcommand{\A}{\mathbf{A}}
\newcommand{\B}{\mathbf{B}}

\newcommand{\D}{\mathbf{D}}
\newcommand{\F}{\mathbf{F}}
\newcommand{\K}{\mathbf{K}}
\renewcommand{\L}{\mathbf{L}}

\newcommand{\N}{\mathbb{N}}
\newcommand{\R}{\mathbf{R}}
\renewcommand{\S}{\mathbf{S}}
\newcommand{\U}{\mathbf{U}}
\newcommand{\V}{\mathcal{V}}
\newcommand{\Z}{\mathbb{Z}}

\newcommand{\Clo}{\mathrm{Clo}}
\newcommand{\Con}{\mathrm{Con}}

\newcommand{\rk}{\mathrm{rk}}

\newcommand{\Cg}{\mathrm{Cg}}
\newcommand{\SMP}{\text{SMP}}

\title[2-Nilpotent Mal'cev Algebras]{Structure and Complexity of 2-Nilpotent Mal'cev Algebras}
\date{\today}
\author[P. Wynne]{Patrick Wynne}
\address{Department of Algebra \\
Charles University \\Prague \\ Czechia}
\email{wynnepm@gmail.com}

\thanks{The research  was partially supported by Charles University under grant number PRIMUS/24/SCI/008 and by Charles University Research Centre
program No. UNCE/24/SCI/022.}



\begin{document}

\begin{abstract}
We investigate the structure of central extensions for algebras in a congruence modular variety.
We use a multisorted algebraic object called a clonoid to understand the term clone 
  of such a central extension. 
We develop the difference clonoid of such a central extension
  and use it to show that 
  the number of $2$-step nilpotent algebras on a fixed finite set is finite
  if and only if the set is of squarefree order.

The subpower membership problem for a finite algebraic structure $\A$ is the problem of deciding 
  on input $a_1,\dots,a_k, b \in A^n$, whether $b$ is in the subalgebra of $\A^n$ generated 
  by $a_1, \dots, a_k$.
We show that for a large class of nilpotent Mal'cev algebras the subpower membership problem 
  is solvable in polynomial time, in particular for $2$-step nilpotent Mal'cev algebras 
  of squarefree order.
\end{abstract} 

\maketitle

%%%%%%%%%%%%%%%%          Section 1:   Introduction          %%%%%%%%%%%%%%%

\section{Introduction}

Deciding whether a given element lies in the substructure generated by a given set of generators
  is one of the oldest and most fundamental questions in algebra.
For vector spaces, this is the question of deciding linear dependence.
For groups this is the famous subgroup membership problem.
Similarly, in polynomial rings the ideal membership problem asks if a polynomial 
  is in the ideal generated by a given set of polynomials. 
In general algebraic structures (or algebras for short)
  the subalgebra membership problem asks to decide if an element of the algebra 
  is obtained by a set of generators using the basic operations of the algebra.
Such membership problems sit at the meeting point between structure and computation;
  naturally, the complexity of the decision problem reflects the structural properties of the algebra.

We consider a variant of the subalgebra membership problem in which we must decide whether a given target tuple 
  lies in the subalgebra of a finite power of an algebra generated by a given set of tuples.
The subpower membership problem for an algebraic structure $\A$, denoted $\SMP(\A)$, 
  is the following computational problem.
On input $k,n \in \N$ and $a_1,\ldots,a_k,b \in A^n$, determine if $b$ is in the subalgebra of $\A^n$ 
  generated by $\{ a_1, \ldots, a_k \}$.
Equivalently, the subpower membership problem is the problem of deciding if 
  a finitary partial function on $A$ can be interpolated by a term function of $\A$. 
For vector spaces, this problem is solvable in polynomial time via Gaussian elimination.
In \cite{Si:CMPG} Sims shows that there is a polynomial time algorithm to determine 
 if a given set of permutations on a finite set generate another given permutation, 
 showing that the subgroup membership problem is tractable (i.e. solvable in polynomial time).
A modification of Sims' method gives a polynomial time algorithm for the subpower membership problem for finite groups. 
For a finite algebra $\A$ of finite type (i.e. with finitely many basic operations), 
  the subpower membership problem problem is decidable in exponential time by enumerating 
  all elements of the subalgebra of $\A^n$ generated by $a_1, \ldots, a_k$, 
  of which there are at most $|A|^n$. 
In \cite{Ko:EXP}, Kozik gives an example of a finite algebra $\A$ of finite type
  such that $\SMP(\A)$ is $\textup{EXPTIME}$-complete, 
  that is, the naive algorithm cannot be improved upon.

Generalizing classical algebraic structures such as groups, rings, modules, and Lie algebras, 
  a Mal'cev algebra is an algebra with a ternary term operation $m(x,y,z)$ satisfying the identities 
  $m(x,y,y) = x = m(y,y,x)$.
In \cite{IMMVW:TL} it is asked whether every finite Mal'cev algebra 
  (more generally, every finite algebra with few subpowers) 
  of finite type has tractable subpower membership problem.
The motivation is a connection to Constraint Satisfaction Problems over certain constraint languages 
  where it is convenient to represent constraint relations by generators.

\begin{que}[\cite{IMMVW:TL} Question 2] \label{que:SMP} 
  Is $\textup{SMP}(\A) \in \textup{P}$ for every finite Mal'cev algebra $\A$ of finite type?
\end{que}

Willard \cite{Wi:SMP} showed that for any expansion of a finite group by multilinear operations, 
  the subpower membership problem is tractable.
In particular it is tractable for finite rings, modules, and $K$-algebras.
In \cite{Ma:SMP}, Mayr shows that for $\A$ a finite Mal'cev algebra of finite type, 
  $\SMP(\A) \in \text{NP}$, and in \cite{BMS:SMP} Bulatov, Mayr, and Szendrei 
  extend this result to show that $\SMP(\A) \in \text{NP}$ for $\A$ a finite algebra 
  of finite type with a \emph{cube term} (equivalently, with few subpowers).
Further they show in \cite{BMS:SMP} that the subpower membership problem is tractable 
  for every finite algebra of finite type with a cube term that generates a residually finite variety. 

Generalizing the concept of nilpotence from classical algebraic structures such as groups and rings,
  the study of nilpotent algebras has been a fecund area of research.
Since finite nilpotent algebras of finite type need not generate residually finite varieties in general,
  this makes them a natural class for which to continue the investigation of Question~\ref{que:SMP}.
A finite group is nilpotent if and only it factors into the direct product of groups of prime power order.
Nilpotent algebras in congruence modular varieties that factor into a direct product of nilpotent algebras
  of prime power order are called supernilpotent.
For a finite supernilpotent Mal'cev algebra $\A$ of finite type, Mayr shows  in \cite{Ma:SMP} that $\SMP(\A) \in \text{P}$. 
However, his methods are not sufficient to show tractability for those finite nilpotent Mal'cev algebras 
  of finite type that are not supernilpotent.
In \cite{Va:NPV} Vaughn-Lee gives an example of a $12$ element $2$-step nilpotent loop 
  that does not factor into the direct product of loops of prime power order, and thus is not supernilpotent. 
Nevertheless, in \cite{Ma:VLL} Mayr shows that this loop has tractable subpower membership problem.
This leads us to the following sub-question of Question~\ref{que:SMP}.

\begin{que} \label{que:SMPnilpotent}
Is $\textup{SMP}(\A) \in \textup{P}$ for every finite $2$-step nilpotent Mal'cev algebra of finite type?
\end{que}

In \cite{Kom:SMP}, Kompatscher makes progress on Question~\ref{que:SMPnilpotent} by showing the following.
\begin{theorem}[\cite{Kom:SMP} Theorem 20] 
  Let $\A$ be a finite Mal'cev algebra with a central series $0_\A < \rho < 1_\A$ such that $|A/\rho| = p$ is a prime 
  and the congruence blocks of $\rho$ are of size coprime to $p$. 
  Then $\textup{SMP}(\A) \in \textup{P}$.
\end{theorem}

In particular, $\SMP(\A) \in \text{P}$ where $\A$ is a finite nilpotent Mal'cev algebra of order $pq$ for primes $p \ne q$. 
We extend this result of Kompatscher to show that a large class of nilpotent Mal'cev algebras 
  that are not necessarily supernilpotent  have tractable subpower membership problem.
We call $\A \cong \L \otimes \U$ a central extension of $\L$ by $\U$. 
See Section~\ref{sec:centralextensions} for details.

\begin{restatable}{theorem}{smpthm}
     \label{thm:SMPmain} 
  Let $\A$ be a finite $2$-nilpotent Mal'cev algebra such that $\A \cong \L \otimes \U$ 
    where $\L$ and $\U$ are abelian Mal'cev algebras of coprime order and $\U$ is the direct product 
    of pairwise non-isomorphic simple Mal'cev algebras. 
  Then $\textup{SMP}(\A) \in \textup{P}$. 
\end{restatable}

As a corollary we obtain the following.

\begin{restatable}{cor}{smpcor}
   \label{thm:SMPcorollary}
  Let $\A$ be a $2$-nilpotent Mal'cev algebra of squarefree order and finite type. 
  Then $\textup{SMP}(\A) \in \textup{P}$.
\end{restatable}

The main tool we will use for the subpower membership problem 
  is a connection between central extensions and clonoids. 
A clonoid from a source algebra to a target algebra is a collection of finitary functions 
  that is closed with respect to composition with term functions of the source algebra on the domain side 
  and term functions of the target algebra on the codomain side. 
We define clonoids more precisely in Section~\ref{Prelims}.
Two algebras with the same underlying set are term equivalent (polynomially equivalent) 
  if their set of term functions (polynomial functions) are the same.
Following \cite{Kom:SMP} we decompose the term clone of a central extension 
  via the \emph{difference clonoid}.
We can understand a $2$-nilpotent Mal'cev algebra by studying an associated clonoid 
  between abelian Mal'cev algebras. 
We develop this connection in Section~\ref{sec:centralextensions} and show that 
  for certain pairs of finite abelian Mal'cev algebras $\L$ and $\U$, 
  there are only finitely many central extensions of $\L$ by $\U$, 
  up to term equivalence (cf. Corollary~\ref{cor:termclonefg}).
In particular, we obtain the following.

\begin{restatable}{theorem}{mainthm}
\label{finitelymany}
     Let $n \in \N$. The number of $2$-nilpotent Mal'cev algebras on $\{1,2, \ldots,n\}$ (up to term equivalence) is finite
     if and only if $n$ is squarefree.
\end{restatable}

The case of $n \in \N$ non-squarefree for which there are countably infinitely many $2$-nilpotent 
  Mal'cev algebras on $\{ 1, \dots, n \}$ was shown in \cite{Id:CCMO} 
  and re-proved by Peter Mayr and the author using clonoids in \cite{MaWy:CBM}. 
This theorem motivates the following question.

\begin{que} \label{que:nilpotent}
  For $n \in \N$ squarefree, is the number of $k$-nilpotent Mal'cev algebras on $\{1,2, \dots, n \}$ 
  (up to term equivalence) finite for $k \ge 3$?
\end{que}

Answering \ref{que:nilpotent} using similar techniques as those used here for Theorem~\ref{finitelymany}
  would likely require a better understanding of clonoids from nilpotent non-abelian 
  Mal'cev algebras into abelian Mal'cev algebras.

The paper is structured as follows.
In Section~\ref{Prelims} we lay out the necessary background and preliminaries from universal algebra.
In Section~\ref{sec:centralextensions} we investigate central extensions in difference term varieties.
We develop the difference clonoid and use it to prove Theorem~\ref{finitelymany}.
In Section~\ref{section:SMP} we again use the difference clonoid for $2$-nilpotent algebras 
to prove Theorem~\ref{thm:SMPmain} and Corollary~\ref{thm:SMPcorollary}.

The results of this paper were originally presented in the PhD thesis of the author (\cite{Wy:CNMA}).

%%%%%%%%%%%%%%%     Section 2:  Preliminaries   %%%%%%%%%%%%%

\section{Preliminaries and Notation} 
\label{Prelims}

We begin with some preliminaries.
For a detailed exposition to the necessary background from 
universal algebra see \cite{BS:UA} or \cite{Berg:UA}.

Let $\N := \{1,2,3,\dots \}$ denote the set of positive integers.
For $n \in \N$ let $[n] := \{1,2,\dots,n\}$.

A signature (or type) $F$ is a collection of function symbols 
  together with an arity $\text{ar}(f) \in \N$ for each $f \in F$.

An algebra in the signature $F$ (or of type $F$) is a pair $\A = (A,\{f^\A \st f \in F\})$ 
  where $A$ is a non-empty set and $f^\A$ is a $\text{ar}(f)$-ary operation on $A$.
  Each $f^\A$ is called a basic operation of $\A$.
We say $\A$ is an algebra of finite type if $F$ is finite.
A term operation of $\A$ of arity $n$ is an operation $t \colon A^n \rightarrow A$
  that is obtained via composition from the basic operations of $\A$ 
  and from the projection operations $\pi_i^k \colon A^k \rightarrow A, \, (x_1, \dots, x_k) \mapsto x_i$.
The term clone (or simply the clone) of an algebra $\A$ 
  is the collection of all term operations of $\A$.
We denote the clone of $\A$ by $\Clo(\A)$,
  and we denote the $n$-ary term operations of $\A$ by $\Clo^n(\A)$.

A congruence of an algebra $\A$ is the kernel of a homomorphism from $\A$, 
  or equivalently a congruence is an equivalence relation on $A$ 
  that is compatible with all basic operations of $\A$. 
We denote by $\Con(\A)$ the congruence lattice of the algebra $\A$.

A variety of signature $F$ is a class of algebras defined by term equations.
A variety $\mathcal{V}$ is called congruence modular if for every algebra $\A \in \V$
  the congruence lattice $\Con(\A)$ satisfies the modular law, i.e. 
  \[z \leq x \implies x \wedge (y  \vee z) = (x \wedge y) \vee z \,\, \text{ for all } x,y,z \in \Con(\A). \]
A variety $\mathcal{V}$ is called congruence permutable if for all $\A \in \mathcal{V}$ 
  and for all $\alpha, \beta \in \Con(\A)$ we have $\alpha \circ \beta = \beta \circ \alpha$,
  where $\circ$ denotes the relational product.
In \cite{Mal:GTAS} Mal'cev showed that a variety is congruence permutable if and only if
  there is a ternary term $m(x,y,z)$, called a Mal'cev term, such that every $\A \in \mathcal{V}$ satisfies the identities 
  \[m^\A(x,y,y) = x = m^\A(y,y,x). \]
We call $m^\A(x,y,z)$ a Mal'cev operation of $\A$.
We call an algebra a Mal'cev algebra if it generates a congruence permutable variety,
  or equivalently if it has a Mal'cev operation in $\Clo(\A)$.

We refer to \cite{FM:CTC} for a complete exposition to commutator theory for congruence modular varieties.
Throughout we use the (term condition) commutator of two congruences 
  $[ \cdot , \cdot ] \colon \Con(\A) \times \Con(\A) \rightarrow \Con(\A)$
  for $\A$ an algebra from a congruence modular variety.
We say that a congruence $\alpha \in \Con(\A)$ is abelian if $[\alpha, \alpha] = 0_\A$,
  where $0_\A$ is the smallest congruence on $\A$. 
We say that an algebra $\A$ is abelian if $[1_\A, 1_\A] = 0_\A$ 
  where $0_\A$ and $1_\A$ denote the smallest and largest congruences on $\A$, respectively.
For $k\in \N$ we say that $\A$ is $k$-step nilpotent, or $k$-nilpotent if 
  \[ [1_\A, [\dots [1_\A, [1_\A,1_\A]] \dots ]] = 0_\A\]
  where the commutator is iterated $k$ times.
In particular we investigate $2$-nilpotent algebras from a congruence modular variety, which satisfy
  \[ [1_\A, [1_\A,1_\A]] = 0_\A .\]

Let  $A$ and $B$ be nonempty sets and $K \subseteq \bigcup_{n \in \N} B^{A^n}$ 
  a collection of functions from finite powers of $A$ into $B$.
For $n \in \N$ we define $K^{(n)}:= K \cap B^{A^n}$ to be the set of $n$-ary functions in $K$.

\begin{definition}
  Let $\A$ and $\B$ be algbras and let $C \subseteq \bigcup_{n \in \N} B^{A^n}$.
  $C$ is an $(\A,\B)$-clonoid if 
    \[C \Clo(\A) \subseteq C \, \text{ and } \Clo(\B)C \subseteq C, \]
    where juxtaposition represents function class composition, as in \cite{Pi:GTMF}.
  Explicitly, 
    \[ f(g_1,\dots , g_n) \in C \text{ for all } f \in C^{(n)} \text{ and } g_1, \dots, g_n \in \Clo^{m}(\A), \]
    and 
    \[ h(f_1, \dots, f_k) \in C \text{ for all } f_1, \dots, f_k \in C^{(n)} \text{ and } h \in \Clo^k(\B). \]
\end{definition}

Let $F(A,B) = \bigcup_{n \in \N} L^{U^n}$ denote the clonoid of all finitary functions from $A$ to $B$. 
For $G \subseteq F(A,B)$ we denote by 
  $\langle G \rangle_{\A,\B}$ the $(\A,\B)$-clonoid generated by $G$,
  i.e. the smallest $(\A,\B)$-clonoid containing $G$.
We say that $f \colon A^n \rightarrow B$ is an \emph{$n$-ary $(\A,\B)$-minor of $g \in F(A,B)$} 
  (or simply $n$-ary minor if $\A$ and $\B$ are clear from context) 
  if $f \in \langle g \rangle_{\A,\B}$.
We say that an $(\A,\B)$-clonoid $C$ is finitely generated if there is a finite set 
  $G \subset C$ such that $C = \langle G \rangle_{\A,\B}$. 

A module is distributive if its submodule lattice is a distributive lattice.
The Jacobson radical of a unital ring $R$ is the intersection of all maximal ideals of $R$.
The nilpotence degree of an ideal $I \trianglelefteq R$ is the least $n$ 
  such that the product of any $n$ elements from $I$ is zero, if such an $n$ exists.
The following is the main theorem of \cite{MaWy:CBM} 
  which we will use in Section~\ref{sec:centralextensions}.

\begin{theorem} \label{thm:distributive} \cite[Theorem 1.4]{MaWy:CBM}
  Let $\A$ be polynomially equivalent to a finite distributive $\R$-module, let $n$ be the nilpotence degree of the
    Jacobson radical of $\R$, and let $\B$ be polynomially equivalent to an $\S$-module such that $|A|$ is
    invertible in $\S$.
  \begin{enumerate}
    \item 
      Then every clonoid from $\A$ to $\B$ is generated by its $n+1$-ary functions 
        (by its $n$-ary functions if $\A$ is an $\R$-module).
    \item 
      If $\B$ is finite, then there are only finitely many clonoids from $\A$ to $\B$. 
\end{enumerate}
\end{theorem}

%%%%%%%%%%%%%%%     Section 3:  Central Extensions  %%%%%%%%%%%%%

\section{Central extensions} 
\label{sec:centralextensions}

Let $\V$ be a congruence modular variety over the signature $F$. 
Then $\V$ has a ternary \emph{difference term} $d$ 
  such that for all $\A\in \V$ and for all $x,y\in A$,
  \[ d^\A(x,x,y) = y  \quad \text{and} \quad  d^\A(y,x,x) \equiv y \mod [\alpha,\alpha], \]
  where
  \[ \alpha = \Cg_\A(x,y) := \bigcap \{\theta \in \Con(\A) \st (x,y) \in \theta\},\]
  is the smallest congruence relating $x$ and $y$ and $[\alpha,\alpha]$ is the commutator 
  of $\alpha$ with itself. 

Note that if $\L \in \V$ is abelian then $[\theta, \theta] = 0$ for all $\theta \in \Con(\A)$.
So $d^\L(x,x,y) = y$ and $d^\L(x,y,y) = x$ for all $x,y \in L.$ 
That is, $d$ interprets as a Mal'cev operation on $\L$. 
 
Let $\V$ be a congruence modular variety with signature $F$ and difference term $d$. 
Let $\L, \U \in \V$. Suppose that $\L$ is an abelian algebra with associated group 
  $\hat{\L} = (L, +, -, 0)$ 
  so that $d^\L(x,y,z) = x-y+z$ for all $x,y,z \in L$. 
See \cite{FM:CTC}~Chapter~5 for a detailed explanation of the abelian group associated 
  to an abelian Mal'cev algebra.
Suppose for each operation symbol $f \in F$ we are given a map $\hat{f} \colon U^k \rightarrow L$ 
  where $k$ is the arity of $f$. 
Let $\hat{F} := \{\hat{f} \st f \in F\}$. We define a new algebra $\L \otimes^{\hat{F}} \U$ 
  with universe $L \times U$ and basic operations 
  \[f^{\L \otimes^{\hat{F}} \U} \colon (L \times U)^k \rightarrow L \times U, \, (\ell, u) \mapsto (f^\L(\ell) + \hat{f}(u), f^\U(u)) \quad \text{ for } f \in F. \]
Here and throughout we slightly abuse notation and let $\ell = (\ell_1,\dots,\ell_k)$ 
  and $u = (u_1,\dots,u_k)$ for $( (\ell_1,u_1),\dots,(\ell_k,u_k) ) \in (L\times U)^k$.
We note that $\L \otimes^{\hat{F}} \U$ may not be in the variety $\V$. 
 
In \cite{FM:CTC} Freese and McKenzie show that in a congruence modular variety $\V$ every algebra $\A$
  that is isomorphic to $\U$ modulo a central congruence $\alpha$ is of the form $\L\otimes\U$ 
  for some abelian $\L\in \V$.
 
\begin{theorem}[\cite{FM:CTC} Theorems 7.1 \& 7.2] 
  Let $\V$ be a congruence modular variety with signature $F$ and let $\A \in \V$. 
  For $\zeta_\A$ the center of $\A$, let $\U = \A/\zeta_\A$.
  \begin{enumerate}
    \item  
      There is an abelian algebra $\L \in \V$ and a system $\hat{F}$ such that $\A \cong \L \otimes^{\hat{F}}\U$. 
      Moreover, the center of $\L \otimes^{\hat{F}} \U$ is the kernel of the projection onto $U$. 
    \item 
      $\A$ is $2$-nilpotent if and only if $\A \cong \L \otimes^{\hat{F}}\U$ for some system $\hat{F}$ 
        and for $\L, \U$ abelian algebras in $\V$. 
  \end{enumerate}
\end{theorem}
 
We call $\L \otimes^{\hat{F}} \U$ a \emph{central extension} of $\L$ by $\U$. 
We note that $\L \otimes^{\hat{F}} \U$ is sometimes called a wreath product or a 
  semidirect product in the literature. 
We will usually suppress the superscript $\hat{F}$ and simply write $\L \otimes \U$. 
Since $\L$ is an abelian Mal'cev algebra, it is straightforward to show by induction on complexity of terms 
  that for each $k \in \N$, every $k$-ary term $t$ over $F$ induces a term function $t^{\L \otimes \U}$ of $\L \otimes \U$ of the form
  \[t^{\L \otimes \U} \colon (L \times U)^k \rightarrow L \times U, \, (\ell, u) \mapsto (t^\L(\ell) + \hat{t}(u), t^\U(u))\]
  for some $\hat{t}\colon U^k \rightarrow L$. 
We will often abuse notation and write $t^{\L \otimes \U} = (t^\L, t^\U) + \hat{t}$. 
We call $\hat{t} \colon U^k \rightarrow L$ the distortion of the term function $t^{\L \otimes \U}$. 
More generally for any $k$-ary term function $t$ on $\L \otimes \U$ and for any 
  $e \colon U^k \rightarrow L$, we abuse notation and denote by $t + e$ the $k$-ary 
  function on $L \times U$ defined by
  \[(\ell, u) \mapsto (t^\L(\ell)+\hat{t}(u) + e(u), t^\U(u)).\]
 
\begin{definition} \label{def:elprime} 
  For $\V$ a congruence modular variety and $\L \in \V$ an abelian algebra we define a new algebra $\L'$ 
    of signature $F$ with universe $L$ and basic operations 
    \[f^{\L'}(\ell_1, \ldots, \ell_k) := f^\L(\ell_1, \ldots, \ell_k) - f^\L(0, \ldots, 0)\]
    for $f \in F$ of arity $k$ and for $0 \in L$.
\end{definition}

Note that $\L' = \L$ if and only if $\{0\}$ is a subuniverse of $\L$. 
Moreover, for $t$ an $F$-term of arity $k$,  by induction on the complexity of terms, we see that 
  \[t^{\L'}(\ell_1, \ldots,\ell_k) = t^\L(\ell_1, \ldots, \ell_k) - t^\L(0, \ldots, 0).\] 
We note that $\L$ and $\L'$ are polynomially equivalent and have the same idempotent term functions. 
Moreover, we see that $\L' \in \V$ as follows.
Let $s$ and $t$ be $F$-terms and suppose that $s^\L = t^\L$ is a $k$-ary identity of $\L$.
Then $s^\L(0, \dots, 0) = t^\L(0, \dots, 0)$ for $0 \in L$,
  and so $s^\L(x) - s^\L(0, \dots, 0) = t^\L(x) - t^\L(0, \dots,0)$ for all $x \in L^k$. 
That is, $s^{\L'} = t^{\L'}$ is an identity of $\L'$.
So $\L'$ satisfies all identities satisfied in $\L$, and hence $\L' \in \V(\L)$. 
In particular, since $\L$ is abelian, $\L'$ is also abelian with Mal'cev term 
  $d^{\L'}(x,y,z) = d^\L(x,y,z) = x-y+z$.
We will often consider the algebra $\L'$ expanded by $0$, which we denote by $\L'_0$.
 
For $f\in F$ define $f'(u) := \hat{f}(u)+f^\L(0,\dots,0)$. 
Let $F' := \{f'\st f\in F\}$. 
Then
  \[ \L\otimes^{\hat{F}}\U = \L'\otimes^{F'}\U. \] 
 
The easiest example of a central extension of $\L$ by $\U$ is the direct product, $\L \times \U$. 
In this case, $\hat{f} = 0$ for every $f \in F$. 
 
For clones $C$ and $D$, we say that a map $\phi \colon C \rightarrow D$ is a \emph{clone homomorphism} if $\phi$
  \begin{enumerate}
    \item preserves arities, i.e. the arity of $f$ equals the arity of $\phi(f)$ for all $f \in C$,
    \item preserves projections, i.e. maps projections in $C$ to the corresponding projections in $D$,
    \item and preserves generalized compositions, i.e. \[\phi(f(g_1,\ldots,g_k)) = \phi(f)(\phi(g_1), \ldots, \phi(g_k)) \] 
      for all $f, g_1, \ldots, g_k \in C$ of the appropriate arities. 
  \end{enumerate}
We note that if $\phi \colon C \rightarrow D$ is a bijective clone homomorphism then $\phi^{-1} \colon D \rightarrow C$ 
  is a clone homomorphism.
In this case we call $\phi$ a clone isomorphism.
 
The following lemmas relate arbitrary central extensions to the direct product via a clone homomorphism. 
 
\begin{lemma} \label{lem:clonehom}
  Let $\V$ be a congruence modular variety and let $\L \otimes \U \in \V$ be a central extension 
    of an abelian $\L$ by $\U$. Let $\L'$ be as in Definition~\ref{def:elprime}. 
  The map
    \[ \xi\colon\Clo(\L\otimes\U) \to \Clo(\L'\times\U),\ f^{\L\otimes \U} \mapsto f^{\L'\times\U}, \]
    is a surjective clone homomorphism. 
\end{lemma}

\begin{proof}  
  Note that $\xi$ is well-defined since for $k$-ary $F$-terms $f,g$ with 
    $f^{\L\otimes \U} = g^{\L\otimes \U}$, we have $f^\U=g^\U$ and
    \[ f^\L(\ell)+\hat{f}(u) = g^\L(\ell)+\hat{g}(u) \text{ for all } \ell\in L^k, u\in U^k. \]
  Setting $\ell = (0,\dots,0)$ in the latter, we obtain
    \[ f^\L(0,\dots,0)+\hat{f}(u) = g^\L(0,\dots,0) +\hat{g}(u) \text{ for all } u\in U^k. \]
  Taking the difference in $\L$ of the two equations above we obtain
    \[ f^\L(\ell) - f^\L(0,\dots,0) = g^\L(\ell)-g^\L(0,\dots,0) \text{ for all } \ell\in L^k. \]
  Thus $f$ and $g$ induce the same function on $\L'$ and $\xi$ is well-defined.
 
  From its definition it is immediate that $\xi$ preserves arities, projections, and the composition of functions.  
  Moreover, let $f$ be a term in the signature of $\V$ that induces the term function 
    $f^{\L' \times \U}$ on $\L' \times \U$. 
  Then $f$ induces a term function $f^{\L \otimes \U}$ on $\L \otimes \U$, and 
    $\xi(f^{\L \otimes \U}) = f^{\L' \times \U}$.
  Hence $\xi$ is surjective.
\end{proof}

Note that  for any central extension $\L \otimes \U$, the projection 
  $\pi_U \colon \L \otimes \U \rightarrow \U$ is a homomorphism.

\begin{lemma} \label{lem:injective} 
  Let $\L \otimes \U$  be as in Lemma~\ref{lem:clonehom} and let $\L'$ be as in Definition~\ref{def:elprime}.
  The projection $\pi_L \colon \L \otimes \U \rightarrow \L'$ is a homomorphism if and only if 
    $\L \otimes \U = \L' \times \U$. 
\end{lemma}

\begin{proof}
  Suppose that $\pi_L \colon \L \otimes \U \rightarrow \L'$ is a homomorphism.
  Let $s$  be an $F$-term and let $\ell \in L^k$ and $u \in U^k$.
  Then 
    \[ \pi_L(s^{\L \otimes \U}(\ell,u))  = \pi_L(s^\L(\ell) + \hat{s}(u), s^\U(u))  = s^\L(\ell) + \hat{s}(u). \]
  Since $\pi_L$ is a homomorphism, 
    \[\pi_L(s^{\L \otimes \U}(\ell,u)) = s^{\L'}(\pi_L(\ell, u)) = s^{\L'}(\ell)= s^\L(\ell) - s^\L(0, \ldots,0).\]
  So $\hat{s}(u) = -s^\L(0, \ldots,0)$.
  Therefore,
    \[ s^{\L \otimes \U} = (s^\L + \hat{s}, s^\U) = (s^\L - s^\L(0, \ldots, 0), s^\U) = s^{\L' \times \U}. \]
  Hence $\L \otimes \U = \L' \times \U$.
  Conversely if $\L \otimes \U = \L' \times \U$ then clearly $\pi_L$ is a homomorphism. 
\end{proof}

The following theorem will show that $\L \otimes \U$ is abelian when $\xi$ is injective and $\U$ is abelian. 

\begin{theorem}\cite[Theorem 9.3]{ALV:III} \label{ALV:clone} 
  Suppose $\A$ and $\B$ are two algebras of signature $F$. 
  \begin{enumerate}
    \item 
      $\B \in \V(\A)$ if and only if there is a clone homomorphism 
        \[\phi \colon \Clo(\A) \rightarrow \Clo(\B)\]
        with $\phi(f^\A) = f^\B$ for all $f \in F$.
    \item 
      $\V(\A) = \V(\B)$ if and only if there exists a clone isomorphism 
        \[\phi \colon \Clo(\A) \rightarrow \Clo(\B)\]
        with  $\phi(f^\A) = f^\B$ for all $f \in F$.
  \end{enumerate}
\end{theorem}

\begin{cor} 
  Let $\xi$ and $\L \otimes \U$  be as in Lemma~\ref{lem:clonehom} and suppose that $\xi$ is injective. 
  For $k \in \N$, if $\U$ is $k$-(super)nilpotent, then $\L \otimes \U$ is $k$-(super)nilpotent.
  In particular if $\U$ is abelian, then $\L \otimes \U$ is abelian. 
\end{cor}
\begin{proof}
  Let $\L'$ be as in Lemma~\ref{lem:clonehom}.
  Since $\xi$ is a injective and surjective, 
    $\xi \colon \Clo(\L \otimes \U) \rightarrow \Clo(\L' \times \U)$ is a clone isomorphism.
  By Theorem~\ref{ALV:clone}, $\V(\L \otimes \U) = \V(\L' \times \U)$.
  In particular, $\L \otimes \U \in \mathbb{HSP}(\L' \times \U)$.
  Since quotients, subalgebras, and products of $k$-(super)nilpotent algebras in a congruence modular variety 
  remain $k$-(super)nipotent, the result follows. 
\end{proof}

\begin{exa} 
  We show that for abelian $\L$ and $\U$, the central extension $\L \otimes \U$ may be abelian even if $\xi$ is not injective. % 
  Let $\L = (\Z_2, +) = \U$. 
  Define $\L \otimes \U = (L \times U, p)$ 
    where 
    \[p^{\L \otimes \U}((\ell_1, u_1) , (\ell_2, u_2))= (\ell_1 + \ell_2 + \hat{p}(u_1,u_2), u_1 + u_2)\]
    with 
    \[ \hat{p} \colon U^2 \rightarrow L, \quad  (u_1, u_2) \mapsto 
      \begin{cases} 
        1  & \text{ if } u_1 = u_2 = 1, \\ 
        0 & \text{ else.} 
      \end{cases}\]
  Then $\L \otimes \U$ is a cyclic abelian group, and hence isomorphic to $(\Z_4, +)$.

  Note that $\xi \colon \Clo(\L \otimes \U) \rightarrow \Clo(\L \times \U)$ is not injective here. 
  So for abelian $\L$ and $\U$, we see that $\L \otimes \U$ may be abelian even if  $\L \otimes \U \ncong \L' \times \U$.
\end{exa}

To capture the situation when $\xi$ is not injective we define the \emph{difference clonoid} of $\L\otimes\U$ as follows.
 
\begin{definition} \label{def:differenceclonoid} 
  Let $\V$ be a congruence modular variety and let $\L \otimes \U$ be a central extension of the abelian $\L$ by $\U$.
  Let  $\L'$ be as in Definition~\ref{def:elprime}.
  For $0 \in L$ we define the \emph{difference clonoid} from $\U$ to $\L'_0$, the algebra $\L'$ expanded by $0$, as
    \[ D(\L\otimes\U) := \{ e\colon U^k\to L \st k\in\N, \,\, x_1+e \in\Clo_k(\L\otimes\U) \}. \]
\end{definition}

We prove that $D(\L\otimes\U)$ is indeed a clonoid in the following.
 
\begin{lemma} \label{lem:diff} 
  Let $\V$ be a congruence modular variety and let $\L \otimes \U \in \V$ be a central extension of $\L$ by $\U$. 
  Let $0 \in L$ and let $\xi$ be as in Lemma~\ref{lem:clonehom}.  
  Let  $\L'$ be as in Definition~\ref{def:elprime}.
  \begin{enumerate}
    \item  \label{it:addon}
      For $k\in\N$ and $e\colon U^k\to L$ the following are equivalent:
      \begin{enumerate} 
        \item $e\in D(\L\otimes\U)$;
        \item $f+e \in\Clo_k(\L\otimes\U)$ for all $f\in\Clo_k(\L\otimes\U)$;
        \item $f+e \in\Clo_k(\L\otimes\U)$ for some $f\in\Clo_k(\L\otimes\U)$;
        \item $(f, f+e) \in \ker(\xi)$ for all $f \in \Clo_k(\L \otimes \U)$.
      \end{enumerate}   
    \item \label{it:clonoid}
      $D(\L\otimes\U)$ is a clonoid from $\U$ to $\L'$ expanded with $0$, in particular to $(L,+,-,0)$.
    \item \label{it:zero} 
      $\xi$ is a clone isomorphism if and only if
        \[D(\L\otimes\U) = \{ f \colon U^k \rightarrow L \st k \in \N, \,  f(x) = 0 \text{ for all } x \}.\]
  \end{enumerate}
\end{lemma}

\begin{proof}
  Let $D := D(\L\otimes\U)$, and let $d$ be a difference term for $\V$.
  Since $\L$ is abelian, $d^\L(x,y,z) = x-y+z$.
  Then \[d^{\L \otimes \U}((\ell_1,u_1),(\ell_2,u_2),(\ell_3,u_3)) = (\ell_1-\ell_2+\ell_3 + \hat{d}(u_1,u_2,u_3), d^\U(u_1,u_2,u_3)).\]
  Since $d$ is a difference term, $\hat{d}(u_1,u_1,u_2) = 0 = \hat{d}(u_2,u_1,u_1)$ for all $u_1,u_2\in U$.

  \eqref{it:addon}
    Let $f,g\in\Clo_k(\L\otimes\U)$ such that $f+e\in\Clo_k(\L\otimes\U)$. 
    We claim 
      \begin{equation} \label{eq:dfe}
        d^{\L\otimes\U}(f+e,f,g) = g+e.
      \end{equation}

  To show~\eqref{eq:dfe} let $\ell\in L^k,u\in U^k$ and consider 
    \begin{align*}
      d^{\L\otimes\U}&(f+e,f,g) (\ell,u)  \\
      = & (d^\L(f^\L(\ell)+\hat{f}(u)+e(u),f^\L(\ell)+\hat{f}(u),g^\L(\ell)+\hat{g}(u)) \\
      & \quad \quad +\underbrace{\hat{d}(f^\U(u),f^\U(u),g^\U(u))}_{=0},  d^\U(f^\U(u),f^\U(u),g^\U(u)) ) \\
      = & (g^\L(\ell)+\hat{g}(u)+e(u), g^\U(u)) \\
      =  & (g+e)(\ell,u). 
    \end{align*}
  All equivalences in \eqref{it:addon} now follow from~\eqref{eq:dfe}.

  \eqref{it:clonoid}
    From~\eqref{it:addon} it follows that $D$ is closed under $+,-,0$ on $L$. 
    More generally, let $f\in F$ be $k$-ary  and
      let $e_1,\dots,e_k\in D$ be $n$-ary. 
    Let $x_i = (\ell_i, u_i) \in L \times U$ for $1 \le i \le n$ and let $u = (u_1, \dots, u_n) \in U^n$.
    Since $\L$ is abelian, we have
    \begin{align*}
      f^{\L\otimes \U} & (x_1+e_1(u), \dots, x_1 +e_k(u)) \\
      &= ( f^\L(\ell_1 + e_1(u), \dots \ell_1 + e_k(u)) + \hat{f}(u_1, \dots, u_1) , f^\U(u_1, \dots, u_1) ) \\
      & =  ( f^\L(\ell_1 - 0 + e_1(u), \dots , \ell_1 - 0 + e_k(u)) + \hat{f}(u_1, \dots, u_1) , \\
      & \quad \quad f^\U(u_1, \dots, u_1) ) \\ 
      & = ( f^\L(\ell_1, \dots, \ell_1) - f^\L(0, \dots, 0)  \\
      & \hspace{1in}+ f^\L(e_1(u), \dots, e_k(u)) + \hat{f}(u_1, \dots, u_1), \\
      & \quad \quad \quad \quad \quad \quad \quad f^\U(u_1, \dots, u_1))\\
      & =  f^{\L\otimes \U} (x_1, \dots, x_1) - f^\L(0, \dots, 0) + f^\L(e_1(u), \dots, e_k(u)) \\
      &  = f^{\L\otimes \U} (x_1, \dots, x_1)  +f^{\L'} (e_1(u), \dots ,e_k(u)).
    \end{align*}
  Hence by \eqref{it:addon},
    $f^{\L'} (e_1,...,e_k) \in D$.
  Thus $\Clo(\L')D \subseteq D$.

  To see $D\Clo(\U) \subseteq D$, let $e\in D$ be $k$-ary, and let $f_1,\dots,f_k\in\Clo(\L\otimes\U)$ be $n$-ary.
  Then $x_1+e$ and consequently $(x_1+e)(f_1,\dots,f_k)$ is a term function of $\L\otimes\U$.
  For $\ell\in L^n,u\in U^n$ we see that
    \begin{align*}
      (x_1+e)(f_1,\dots,f_k)(\ell,u) & = (f^\L_1(\ell)+ e(f_1^\U(u),\dots,f^\U_k(u)), f_1^\U(u) ) \\
      & = (f_1+ e(f_1^\U,\dots,f^\U_k)) (\ell,u).           
    \end{align*}
  From~\eqref{it:addon} it follows that $e(f_1^\U,\dots,f^\U_k)  \in D$ and $D$ is a clonoid.
 
  \eqref{it:zero} 
    Follows from \eqref{it:addon}.
\end{proof}

Next we determine a set of generators for the term clone of $\L \otimes \U$. 

\begin{theorem} \label{thm:clonegens}
  Let $\L,\U$ be algebras in a congruence modular variety $\V$ with difference term $d$, 
    let $\L\otimes\U$ be a central extension of $\L$ by $\U$ and let $\xi$ be the clone 
    homomorphism from $\Clo(\L\otimes\U)$ onto $\Clo(\L'\times\U)$ from Lemma~\ref{lem:clonehom}.
 
  Let $G \subseteq \Clo(\L\otimes\U)$ such that $\xi(G)$ generates $\Clo(\L'\times\U)$.
  Let $0 \in L$ and let $E$ be a generating set of
    the clonoid $D(\L\otimes\U)$ from $\U$ to $(L, +, -, 0)$.
  Then
    \[ \Clo(\L\otimes\U) = \langle \{d^{\L\otimes\U}\} \cup G \cup (x_1+E) \rangle. \]
\end{theorem}

\begin{proof}
  The inclusion $\supseteq$ is clear since $S :=  \{d^{\L\otimes\U}\} \cup G \cup (x_1+E)$ 
    is a set of term functions of $\L\otimes\U$.
  For the converse inclusion let $t\in\Clo(\L\otimes\U)$. 
  By assumption we have $s\in\langle G\rangle$ such that $\xi(s) = \xi(t)$. 
  Hence by Lemma~\ref{lem:diff} we have $e\in D(\L\otimes\U)$ such that
    \[ t=s+e = d^{\L\otimes\U}(x_1+e,x_1,s). \]
  Thus $t\in\langle S\rangle$ follows once we have proved that
    \begin{equation} \label{eq:x1+D} 
      x_1+D(\L\otimes\U) \subseteq \langle S\rangle.
    \end{equation}
  For this we show that
    \[ D := \{e\in D(\L\otimes\U) \st x_1+e \in \langle S\rangle \} \]
    is a clonoid from $\U$ to $(L,+,-,0)$. 
  Clearly $0\in D$. For $e_1,e_2\in D$ we see from
    \[ d^{\L\otimes\U}(x_1+e_1, x_1, x_1+e_2) = x_1+(e_1+e_2),\quad d^{\L\otimes\U}(x_1,x_1+e_1,x_1) = x_1-e_1 \]
    that $D$ is closed under $+,-$ on the outside.

  Next let $e\in D$ be $k$-ary and let $p_1^\U,\dots,p_k^\U$ be $n$-ary term functions of $\U$. 
  Then there exist $n$-ary functions $q_1,\dots,q_k\in\langle G\rangle$ such that
    \[q_{i}(\ell,u) = (p_{i}^\L(\ell)+ \hat{p}_{i}(u), p_{i}^\U(u)) \text{ for all } i\in [k],\ell\in L^n,u\in U^n.\]
  Moreover, as in the proof of Lemma~\ref{lem:diff} item \ref{it:clonoid}, we see that
    \[ (x_1+e)(q_1,\dots,q_k) = q_1+e(q_{1},\dots,q_{k}) = q_1+e(p_1^\U,\dots,p_k^\U). \]
  Hence $D \Clo(\U) \subseteq D$. Thus $D$ is a clonoid from  $\U$ to $(L,+,-,0)$. Since $E\subseteq D$ by definition,
    $D=D(\L\otimes\U)$ and~\eqref{eq:x1+D} is proved. The result follows.
\end{proof}

\begin{cor} \label{cor:fingen} 
  Let $\V$ be a congruence modular variety and let $\L \otimes \U \in \V$ be a central 
    extension of $\L$ by $\U$.
  If $\Clo(\L' \times \U)$ is finitely generated and $D(\L \otimes \U)$ is finitely generated 
    then $\Clo(\L \otimes \U)$ is finitely generated. 
\end{cor}

Combining this with Theorem~\ref{thm:distributive} and the observation that finite abelian Mal'cev 
  algebras have term clones generated by their ternary functions we obtain the following.

\begin{cor} \label{cor:termclonefg} 
  Let $\V$ be a congruence modular variety and let $\L \otimes \U \in \V$ be a central extension 
    of $\L$ by $\U$ such that $L$ and $U$ are of coprime order. 
  Further suppose that $\U$ is polynomially equivalent to a distributive module over a ring whose 
  Jacobson radical has nilpotence degree $n$.
  Then $\Clo(\L \otimes \U)$ is generated by its functions of arity $\text{max}(3, n+1)$.
\end{cor}

Hence we obtain the main theorem of this section.

\mainthm*

\begin{proof} 
  If $n$ is not squarefree the number of $2$-nilpotent Mal'cev algebras is countably infinite.
  This follows from \cite{Id:CCMO}.

  Let $n$ be squarefree and let $\A$ be a $2$-nilpotent Mal'cev algebra of order $n$. 
  Then $\A \cong \L \otimes \U$ for abelian Mal'cev algebras $\L$ and $\U$. 
  Since $\L$ and $\U$ are abelian, $\Clo(\L)$ and $\Clo(\U)$ are both generated by their ternary functions. 
  Hence $\Clo(\L' \times \U)$ is generated by its ternary functions.
  Since $n$ is squarefree, $|L|$ and $|U|$ are coprime and squarefree. 
  Hence $\U$ is polynomially equivalent to a finite distributive module over a direct 
    product of finite fields of prime order. 
  By Theorem~\ref{thm:distributive} every clonoid from $\U$ to $\L$ is generated by its binary functions.
  In particular, $D(\L \otimes \U)$ is generated by its binary functions. 
  By Corollary~\ref{cor:fingen}, $\Clo(\L \otimes \U)$ is generated by its ternary term functions.
  As there are only finitely many ternary functions on a finite set, there are finitely 
    many nilpotent Mal'cev clones on $\{1,2,\ldots,n\}$. 
\end{proof}

Since nilpotent algebras in a congruence modular variety are congruence uniform, 
  the nilpotence degree of an algebra of squarefree order $n$ is at most the number of 
  prime divisors of $n$. 
So we obtain the following corollary.

\begin{cor} 
  Let $p \neq q$ be primes. 
  Then the number of nilpotent Mal'cev algebras on $\{1,2,\ldots,pq\}$ (up to term equivalence) is finite.
\end{cor}

%%%%%%%%%%%%%%%     Section 4:  Subpower Membership Problem  %%%%%%%%%%%%%

\section{The Subpower Membership Problem}\label{section:SMP}

The subpower membership problem for an algebra $\A$, denoted $\SMP(\A)$, 
  asks if a given partial function from $A^n$ to $A$ can be interpolated by a term function of $\A$. 
An equivalent formulation of $\SMP(\A)$ is as follows:

\noindent $\text{SMP}(\A)$:

\noindent  \quad Input: $k, n \in \N$ and $a_1, \ldots, a_k, b \in A^n$.

\noindent \quad Problem: Decide if $ b \in \langle a_1,\ldots,a_k \rangle_{\A^n}$, 
                          the subalgebra of $\A^n$ generated by $\{a_1, \dots, a_k\}$.

The main step in finding a polynomial time algorithm for $\SMP(\A)$ 
  with $\A$ as in Theorem~\ref{thm:SMPmain} is a polynomial time reduction to the problem of 
  computing a compact representation of a particular subalgebra of a power of $\A$.
This is the prevailing strategy in, e.g. \cite{Ma:SMP} and \cite{BMS:SMP}.

Let $\A$ be a finite Mal'cev algebra, let $n \in \N$  and let $R \subseteq A^n$. 
Following \cite{BD:ASA}, we define the signature of $R$, denoted $\text{Sig}(R)$, as the set of 
  all triples $(i, a, b) \in \{1,2, \ldots, n\} \times A^2$ 
  such that there exist $t_a, t_b \in R$ with $t_a(j) = t_b(j)$ for all $j < i$ 
  and $t_a(i)=a, \, t_b(i)=b$. 
We say that $t_a$ and $t_b$ are witnesses for $(i,a,b) \in \text{Sig}(R)$. 
In \cite{BD:ASA}, Bulatov and Dalmau show that 
  for $R \le \A^n$, if $S \subseteq R$ is such that $\text{Sig}(S) = \text{Sig}(R)$, 
  then $S$ generates $R$. 
In fact, they prove that $R$ is the closure of $S$ with respect to just the Mal'cev term $m$ of $\A$. 

For $R \le \A^n$, we call $S \subseteq R$ a representation of $R$ if $\text{Sig}(S) = \text{Sig}(R)$, 
  and we call $S$ a compact representation of $R$ if $\text{Sig}(S) = \text{Sig}(R)$ 
  and $|S| \le 2|\text{Sig}(R)|.$ 
It can be shown that every subalgebra of a power of a finite Mal'cev algebra has a compact representation. 

In order to solve $\SMP(\A)$ in polynomial time, it suffices to compute a compact representation 
  of $R:= \langle a_1, \dots, a_k \rangle_{\A^n}$  in time polynomial in $n$ and $k$. 
So, following \cite{BD:ASA}, we define the following problem for a finite Mal'cev algebra $\A$.

\noindent $\text{CompRep}(\A)$:

\noindent \quad  Input: $a_1, \ldots, a_k \in A^n$

\noindent \quad Output: A compact representation of $\langle a_1,\ldots,a_k \rangle_{\A^n}.$

\begin{theorem}[cf. \cite{BD:ASA}] For $\A$ a finite Mal'cev algebra, $\textup{SMP}(\A)$ is polynomial time reducible to $\textup{CompRep}(\A).$
\end{theorem}

We note that to solve $\text{CompRep}(\A)$ on input $a_1, \dots, a_k \in A^n$ 
  it suffices to find a representation 
  $S \subseteq R:= \langle a_1, \dots , a_k \rangle_{\A^n}$ 
  with $\text{Sig}(S) = \text{Sig}(R)$ which is of size  polynomial in $n$ and $k$. 
We can then pass from this representation to a compact representation $S' \subseteq S$ of $R$ 
  in time polynomial in $n$ and $k$.

Following Kompatscher \cite{Kom:SMP} we define a version of the compact representation problem 
  for clonoids. 
Let $C$ be a clonoid from algebra $\U$ to algebra $\L$.

\noindent $\text{CompRep}(C)$:

\noindent \quad Input: $a_1, \ldots, a_k \in U^n$.

\noindent  \quad Output: A compact representation of 
\[C( a_1,\ldots,a_k):= \{f(a_1, \ldots, a_k) \st f \in C^{(k)} \} \le \L^n.\]

In \cite{Kom:SMP} Kompatscher shows that for certain central extensions we can further reduce 
  the problem $\text{CompRep}(\A)$ to $\text{CompRep}(D)$ for $D$ the difference clonoid of $\A$.

\begin{theorem}\cite[Theorem 11]{Kom:SMP}  \label{thm:comprep}
Let $\A = \L \otimes \U$ be a finite Mal'cev algebra such that $\U$ is supernilpotent. 
Then $\textup{SMP}(\A)$ reduces in polynomial time to $\textup{CompRep}(D(\L \otimes \U))$.
\end{theorem}

So in order to efficiently solve $\SMP(\L \otimes \U)$ it suffices to find a generating set for 
  $D(\L \otimes \U)(a_1, \ldots, a_k) \le \L^n$ of size polynomial in $n$ and $k$. 
To do so for our desired setting we use the interpolation arguments 
  for clonoids between abelian Mal'cev algebras that were used for the proof of \ref{thm:distributive}.
We first develop the necessary technical lemmas. 

\begin{lemma}  \label{lem:dps}
 Let $\ell \in \N$ and let $\A_1, \ldots, \A_\ell$ be one-dimensional vector spaces over finite fields $\F_1, \ldots, \F_\ell$, respectively.
 Let $\A = \A_1 \times \cdots \times \A_\ell$ be a module over $\F:= \F_1 \times \cdots \times \F_\ell$. 
 Let $\B$ be a finite $\S$-module such that $|A|$ and $|B|$ are coprime.
 Let $k\in\N$ and $N\leq\A^k$ isomorphic to $\A$.
 Then there exists $s\colon \{r\in F^{k\times k} \st \rk(r) \leq 1 \} \to S$ such that for every $f\colon A^{k}\to B$
  with $f(0,\dots,0)=0$, the function 
  \[ f_N(x) := \sum_{r\in R_\A^{k\times k}, \rk(r)\leq 1} s(r) f(rx) \]
  satisfies 
  \[f_N(x) = \begin{cases} f(x) & \text{if } x \in N, \\ 0 & \text{else.} \end{cases} \]. 
\end{lemma}
\begin{proof} 
  Follows from \cite[Corollary 3.5]{MaWy:CBM}.
\end{proof}

For $i \in \{1, \ldots, \ell \}$ suppose $\A_i$ is a finite simple abelian Mal'cev algebra. 
Denote by ${\A_i}_0$ the algebra obtained by expanding $\A_i$ by constant $0 \in A_i$.
Then ${\A_i}_0$ is a simple $\R_{\A_i}$-module. 
By the Wedderburn-Artin Theorem $\R_{\A_i} \cong \K_i^{n_i\times n_i}$ 
  and ${\A_i}_0 \cong \K_i^n$ for some finite field $\K_i$ and $n_i \in\N$.
Note that that the matrix ring $\K_i^{n_i\times n_i}$ (and hence $\R_{\A_i}$) 
  has a subring $\F_i$, which is a field of order $|K_i|^{n_i}$, and ${\A_i}_0$ has 
  a reduct $\hat{\A}_i$, which is a one-dimensional vector space over $\F_i$.

\begin{lemma} \label{lem:malcev}
 Let $\A = \A_1 \times \cdots \times \A_\ell$ be a product of finite simple abelian Mal'cev algebras 
  that are pairwise non-isomorphic.
 Let $\B$ be a finite $\S$-module such that $|A|$ and $|B|$ are coprime.
 
 Let $\F_i \leq\R_{\A_i}$ be a field of order $|F_i|=|A_i|$ for $1 \le i \le \ell$.
 Let $0_i \in A_i$ and $\hat{\A}_i$ be a polynomial reduct of $\A_i$ that is
  a $1$-dimensional vector space over $\F_i$ with zero $0_i$.
 Let $\F := \F_1 \times \cdots \times \F_\ell$ 
  and let $\hat{\A} := \hat{\A}_1 \times \dots \times \hat{\A}_\ell$ be an $\F$-module 
  with zero $0:= (0_1, \ldots, 0_\ell)$. 
 Let $k\in\N$, $D := \{ (z,\dots,z) \in A^{k+1} \st z\in A \}$, and $D\leq N\leq\hat{\A}^{k+1}$ 
  such that $N\cong\hat{\A}^2$.
 
  Then there exists $s\colon \{r\in F^{k\times k} \st \rk(r) \leq 1 \} \to \S$ such that for all
    $f\colon A^{k+1}\to B$ satisfying $f(z,\dots,z)=0$ for all $z\in A$
    the function
    \[ f_N(x,z) := \sum_{r\in F^{k\times k}, \rk(r)\leq 1} s(r) f(r*_zx,z) \text{ for } x\in A^k, z\in A, \]
    satisfies
    \[ f_N(x,z) = 
      \begin{cases} 
        f(x,z) & \text{if } (x,z)\in N, \\
        0 & \text{else}. 
      \end{cases} 
    \]
\end{lemma}

\begin{proof}
  Note that $\hat{\A}$ and $\B$ satisfy the assumptions of Corollary~\ref{lem:dps}.
  Since $\pi_i(N)$ has dimension $2$ over $\F_i$ for all $1 \le i \le \ell$, 
    we have $N = F*_0 a+_0 D$ for some $a\in A^k\times \{0\}$. 
  In particular
    \[ N = \{r(a,0) \st r\in F \} +_0 D = \{r(a,d) \st r\in F, d\in D \} = \bigcup_{z\in A} F*_z a. \]
  For $z\in A$, the $\F$-module isomorphism $p_z\colon A \to A_z,\ x\mapsto x+_0z$, 
    naturally induces an isomorphism 
    $\hat{\A}^k \times \{0\} \to \hat{\A}_z^k \times \{z\}$ 
    that maps the submodule $N_0 := F *_0 a$ to
    $N_z := F *_z p_z(a)$. Further $N$ is the disjoint union of all $N_z$ for $z\in A$.

  By Corollary~\ref{lem:dps} there exists $s\colon \{r\in F^{k\times k} \st \rk(r)\leq 1 \} \to S$ 
    such that for all $f\colon A^{k+1}\to B$ with $f(0,\dots,0)=0$, 
    \begin{equation} \label{eq:f0N0}
      f_{N_0}(x,0) := \sum_{r\in F^{k\times k}, \rk(r)\leq 1} s(r) f(r*_0x,0)
    \end{equation}
    satisfies
    \[ f_{N_0}(x,0) = 
      \begin{cases} 
        f(x,0) & \text{if } (x,0)\in N_0, \\
        0 & \text{else if } (x,0) \in A^k\times\{0\}. 
      \end{cases} 
    \]

  Let $f\colon A^{k+1}\to B$ such that $f(z,\dots,z)=0$ for all $z\in A$.
  Let $z\in A$. 
  Abusing notation we write
    \[ fp_z \colon A^{k+1}\to A,\ (x_1,\dots,x_{k+1}) \mapsto f(p_z(x_1),\dots,p_z(x_{k+1})). \]
  Since $fp_z(0,\dots,0) = f(z,\dots,z) = 0$, we may apply~\eqref{eq:f0N0} to $fp_z$ instead of $f$.
  Using that $p_z$ is a homomorphism, we obtain for all $x\in A^k$ that
    \begin{align*}
      (fp_z)_{N_0}(x,0) &= \sum_{r\in F^{k\times k}, \rk(r)\leq 1} s(r) f(p_z(r*_0x),p_z(0)) \\
      &= \sum_{r\in F^{k\times k}, \rk(r)\leq 1} s(r) f(r*_zp_z(x),z).
    \end{align*}
  Since $p_z$ induces a bijection on $A^{k+1}$ that maps $N_0$ to $N_z$, this yields that
    \[ f_{N_z}(x,z) := \sum_{r\in F^{k\times k}, \rk(r)\leq 1} s(r) f(r*_zx,z) \]
    satisfies
    \[ f_{N_z}(x,z) = 
      \begin{cases} 
        f(x,z) & \text{if } (x,z)\in N_z, \\
        0 & \text{else if } (x,z)\in A^k\times\{z\} . 
      \end{cases} 
    \]
  Since $N$ is the union of $N_z$ for all $z\in A$,
    \[ f_N = \bigcup_{z\in A} f_{N_z} \]
    has the required properties. 
\end{proof}  

From the previous lemma we obtain:

\begin{lemma} \label{lem:fdecomp}
  Let $\ell \in \N$ and let $\A = \A_1 \times \cdots \times \A_\ell$ be a product of 
    finite simple abelian Mal'cev algebras that are pairwise non-isomorphic.
  Let $\B$ be a finite module such that the orders of $\A$ and $\B$ are coprime.

  Let $\F_i \leq\R_{\A_i}$ be a field of order $|F_i|=|A_i|$ for $1 \le i \le \ell$.
  Let $0_i \in A_i$ and $\hat{\A}_i$ be a polynomial reduct of $\A_i$ that is
    a $1$-dimensional vector space over $\F_i$ with zero $0_i$.
  Let $\F := \F_1 \times \cdots \times \F_\ell$ 
    and let $\hat{\A} := \hat{\A}_1 \times \dots \times \hat{\A}_\ell$ be an $\F$-module 
    with zero $0:= (0_1, \ldots, 0_\ell)$. 
  Let $k\in\N$, $0\in A$, $D := \{ (z,\dots,z) \in A^{k+1} \st z\in A \}$
    and
    \[ V := \{ N\leq \hat{\A}^{k+1} \st D \leq N,\, N \cong \hat{\A}^2 \}. \]
  Then for each $f\colon A^{k+1}\to B$ the functions
    \begin{align*}
      f'(x,z) := f(x,z)-f(z,\dots,z,z)  \\
      f'_N(x,z) := 
        \begin{cases} 
          f'(x,z) & \text{if } (x,z) \in N \\
          0 & \text{else}, 
        \end{cases}
    \end{align*}
    are $(\A,\B)$-minors of $f$ and
    \[ f(x,z) = f(z,\dots,z,z) + \sum_{N\in V} f'_N(x,z). \]         
\end{lemma}

\begin{proof}
  By its definition $f'$ is an $(\A,\B)$-minor of $f$ and $f'(z,\dots,z,z)=0$ for all $z\in A$.
  Then Lemma~\ref{lem:malcev} applies to $f'$ and yields that $f'_N$ are $(\A,\B)$-minors of $f'$, 
    and hence of $f$.
  Since the submodules $N\in V$ cover $A^{k+1}$ and any distinct $N,M\in V$ intersect exactly in $D$,
    we have $f' = \sum_{N\in V} f'_N$ and the result follows.
\end{proof}

We are now ready to show that we can efficiently compute a generating set 
for the image of a clonoid in this setting.

\begin{lemma} \label{lem:generators}
  Let $\ell \in \N$ and let $\A = \A_1 \times \A_2 \times \cdots \times \A_\ell$ 
    be a product of pairwise non-isomorphic finite simple abelian Mal'cev algebras $\A_1, \ldots, \A_\ell$.
  Let $\B$ be a finite $\S$-module such that the orders
    of $\A$ and $\B$ are coprime. Let $C$ be a clonoid from $\A$ to $\B$.

  Given $k,n\in\N$ and $a_1,\dots,a_{k+1}\in A^n$, we can compute a set of generators for 
    \[ C^{(k+1)}(a_1,\dots,a_{k+1}) := \{ f(a_1,\dots,a_{k+1}) \st f\in C^{(k+1)} \} \leq \B^n \]
    in time polynomial in $n$ and $k$.
\end{lemma}  
        
\begin{proof}
  We adopt all notation as in Lemma~\ref{lem:malcev} and the comments immediately preceding.

  Let $D := \{ (z,\dots,z)\in A^{k+1} \st z\in A\}$ and
    \[V := \{ N\leq \hat{\A}^{k+1} \st D \leq N, \,  N \cong \hat{\A}^2 \}.\]
  By Lemma~\ref{lem:fdecomp} every $f\in C^{(k+1)}$ is a sum of $f(z,\dots,z,z) \in C^{(1)}$ and a sum of
    functions $f'_N \in C^{(k+1)}$ with support in $N\setminus D$ for some $N\in V$ for 
    \[f'(x_1, \dots, x_k,z) = f(x_1, \dots, x_k, z) - f(z,\ldots,z,z).\]

  Let $N\in V$. Our first goal is to parametrize $N$ by elements in $A^2$ via term functions of $\A$.
  From the first paragraph of the proof of Lemma~\ref{lem:malcev}, recall that we have $N = F*_0 a+_0 D$
    for some $a\in A^k\times \{0\}$.
  Let $m$ be the Mal'cev term of $\A$.
  For $z\in A$, the $\F$-module isomorphism $p_z\colon A_0\to A_z,\ x\mapsto m(x,0,z)$, 
    naturally induces an isomorphism $\hat{\A}_0^k \times \{0\} \to \hat{\A}_z^k \times \{z\}$ 
    that maps the submodule $N_0 := F *_0 a$ to $N_z := F *_z p_z(a)$.

  For $1 \le i \le \ell$, let $\pi_i$ denote the $i^{th}$ projection of $A^{k+1}$ onto $A_i^{k+1}$, 
    that is, 
    \[ \pi_i \colon A_1^{k+1} \times \dots \times A_\ell^{k+1} \rightarrow A_i^{k+1}, \quad (x_1, \dots x_\ell) \mapsto x_i.\]
  Note that $N_0 := F *_0 a$ is such that $\pi_i(N_0)$ has dimension $1$ over $\F_i$ for each $1 \le i \le \ell$.
  So we have $r_1^i,\dots,r_k^i\in F_i$ such that for each $1 \le i \le \ell$,
    \[\pi_i( N_0 )= \{(r_1^i(x,0),\dots,r_k^i(x,0),0) \st x\in A_i \}. \]
  So for $r_j := (r^1_j, \ldots, r^\ell_j) \in F$ for $1 \le j \le k$ we have 
    \[N_0 = \{(r_1(x,0), \ldots, r_k(x,0)) \st x \in A \}. \]
  Since $p_z$ is an $\F$-module isomorphism, we see that
    \begin{align*}
      N_z & = p_z(N_0) \\
      & = \{(r_1(p_z(x),z),\dots,r_k(p_z(x),z),z) \st x\in A \} \\
      & = \{(r_1(x,z),\dots,r_k(x,z),z) \st x\in A \}.
    \end{align*}
  Hence
    \[ N = \bigcup_{z\in A} N_z = \{(r_1(x,z),\dots,r_k(x,z),z) \st x,z \in A \}. \] 
  Define
    \[ q_N\colon A^2\to N,\ (x,z)\mapsto (r_1(x,z),\dots,r_k(x,z),z). \]
  Then $q_N$ is a bijection induced by term functions of $\A$. 
  Consequently,
    \[ g_N := f'q_N \in C^{(2)} \]
    and $g_N(z,z) = 0$ for all $z\in A$. Clearly $g_N q_N^{-1} = f'_N|_N$.

  Let $I := \{ i\in \{1,\dots,n\} \st (a_{1i},\dots,a_{k+1i}) \not\in D \}$.
  For each $i\in I$ let $N_i$ be the unique element in $V$ such that $(a_{1i},\dots,a_{k+1i})\in N_i$.
  For each $g\in C^{(2)}$ with $g(z,z) = 0$ for all $z \in A$ and for each $i\in I$ define $b_{g,i}\in B^n$ by
    \[ (b_{g,i})_j := \begin{cases} gq_{N_i}^{-1}(a_{1j},\dots,a_{k+1j}) & \text{if } j\in I, N_j = N_i \\ 0 & \text{else.} \end{cases} \]
  Then $C^{k+1}(a_1,\dots,a_{k+1}) \le \B^n$ is generated by
    \[ G:= C^{(1)}(a_{k+1}) \cup \{ b_{g,i} \st g\in C^{(2)}, g(z,z) = 0 \text{ for all } z, \, i \in I \}. \]
  To see this, suppose that $f \in C^{k+1}$. 
  Then 
    \[f(x_1, \dots, x_k,z)  = f(z,\dots,z) + \sum_{N \in V} f'_{N}(x_1,\dots, x_k,z)\]
    and $f(z,\dots,z,z) \in C^{(1)}$. 
  Moreover, for each $N \in V$ we have that $g_N q_N^{-1} = f'_N|_N$
    and so 
    \[ \sum_{N \in V} f'_{N}(a_1, \dots, a_{k+1}) = \sum_{i \in I} f'_{N_i}(a_1, \dots, a_{k+1}) = \sum_{i \in I} g_{N_i}q_{N_i}^{-1}(a_1, \dots, a_{k+1}), \]
    which is generated by  $\{ b_{g,i} \st g\in C^{(2)}, g(z,z) = 0 \text{ for all } z, \, i \in I \}$. 
  Hence $G$ generates $f(a_1, \dots, a_{k+1}) \in \B^n$ for each $f \in C^{k+1}$.

  For $N \in V$, $N$ is the submodule of $\A^{k+1}$ generated by the diagonal $\D$ and some 
    element $a \notin D$.
  We have $q_N$ given by $q_N(x,z) = (r_1(x,z), \dots, r_k(x,z),z)$ for $r_1, \dots r_k \in \F$,
    and $r_i = (r_{i}^1, \dots, r_{i}^{\ell})$ with $r_{i}^j \in \F_i$ for each $ 1 \le j \le \ell$. 
  Since each $\F_i$ is a field we can compute $r_i^1, \dots, r_i^\ell$ in time polynomial in $k$ 
    for each $1 \le i \le \ell$. 
  Therefore we can compute $q_N^{-1}$ in time polynomial in $k$ for each $N$. 
  Now, to compute $G$ we need only to compute $gq_{N_i}^{-1}(a_{1,i}, \dots, a_{{k+1},i})$ for each $i \in I$, and $|I| \le n$.
  Hence we need to compute at most $n$ many $q_N$ for $N \in V$.
  Moreover the number of $g \in C^{(2)}$ satisfying $g(z,z) = 0$ is bounded by $|B|^{|A|^2}$. 
  In particular it is independent of the input size. 
  Hence we can compute $G$ in time polynomial in $n$ and $k$. 
\end{proof}

We are now ready to prove the main theorem of this section. 

\smpthm*

\begin{proof}
  By Lemma~\ref{lem:generators}, since $\U$ is a product of pairwise non-isomorphic 
    simple abelian Mal'cev algebras, we can compute a polynomial size generating set $G$ 
    for $D(\L \otimes \U)(a_1, \ldots, a_{k+1}) \le \L^n$ on inputs $a_1, \ldots, a_{k+1} \in U^n$ 
    in time polynomial in $n$ and $k$.
  If this generating set is not a compact representation of $D(\L \otimes \U)(a_1, \ldots, a_{k+1})$, 
    that is,  if $|G| > 2|\text{Sig}(D(\L \otimes \U)(a_1, \ldots, a_{k+1}))|$,
    we can then reduce this generating set in polynomial time to obtain a compact representation of 
    $D(\L \otimes \U)(a_1, \ldots, a_{k+1}) \le \L^n$.
  Since $\U$ is abelian (and hence supernilpotent), $\SMP(\A)$ reduces in polynomial time to 
    $\text{CompRep}(D(\L \otimes \U))$ by Theorem~\ref{thm:comprep}.
  This yields a polynomial time algorithm for $\SMP(\A)$. 
\end{proof}

\smpcor*

It remains open whether every finite $2$-nilpotent Mal'cev algebra of finite type 
  has tractable subpower membership problem.
Ongoing research aims to use the techniques employed here to show that $\SMP(\L \otimes \U)$ is tractable
  for all $2$-nilpotent Mal'cev algebras $\L \otimes \U$.
This approach requires a better understanding of clonoids between finite abelian Mal'cev algebras of coprime order.
In recent work, Fioravanti, Kompatscher, and Rossi (\cite{FKR:CBVS}) show that clonoids from 
  (an algebra polynomially equivalent to) a finite vector space to (an algebra polynomially equivalent to) 
  a finite module of coprime order are all finitely generated and there are finitely many of them (cf. \cite{FKR:CBVS}~Theorem~49).
In light of this result as well as Theorem~\ref{thm:distributive} (\cite[Theorem 1.4]{MaWy:CBM})
  future work attempts to make progress towards the following conjecture.
\begin{conjecture}[cf. \cite{FKR:CBVS}~Conjecture 1]
  For $\U$ and $\L$ abelian Mal'cev algebras of coprime order, every $(\U,\L)$-clonoid is finitely generated
    and there are finitely many such clonoids.
\end{conjecture}
Settling this conjecture will likely prove beneficial to showing that $\SMP(\L \otimes \U) \in \textup{P}$ 
  for all such $\L$ and $\U$.

%%%%%%%%%%%%%%%  Acknowledgements  %%%%%%%%%%%%%%%%%%%%

\section*{Acknowledgments}

 The author thanks Peter Mayr for many helpful discussions that contributed to the results of this paper.

%%%%%%%%%%%%%%%          Bibliography           %%%%%%%%%%%%%%%

%\bibliographystyle{plain} 
%\bibliography{../../../biblio}

\begin{thebibliography}{10}

\bibitem{Berg:UA}
Clifford Bergman.
\newblock {\em Universal algebra}, volume 301 of {\em Pure and Applied
  Mathematics (Boca Raton)}.
\newblock CRC Press, Boca Raton, FL, 2012.
\newblock Fundamentals and selected topics.

\bibitem{BD:ASA}
Andrei Bulatov and V\'{\i}ctor Dalmau.
\newblock A simple algorithm for {M}al'tsev constraints.
\newblock {\em SIAM J. Comput.}, 36(1):16--27, 2006.

\bibitem{BMS:SMP}
Andrei Bulatov, Peter Mayr, and \'{A}gnes Szendrei.
\newblock The subpower membership problem for finite algebras with cube terms.
\newblock {\em Log. Methods Comput. Sci.}, 15(1):Paper No. 11, 48, 2019.

\bibitem{BS:UA}
Stanley Burris and H.~P. Sankappanavar.
\newblock {\em A course in universal algebra}, volume~78 of {\em Graduate Texts
  in Mathematics}.
\newblock Springer-Verlag, New York-Berlin, 1981.

\bibitem{FKR:CBVS}
Stefano Fioravanti, Michael Kompatscher, and Bernardo Rossi.
\newblock Clonoids between vector spaces.
\newblock {\em International Journal of Unpublished Manuscripts}, 2025.

\bibitem{FM:CTC}
R.~Freese and R.~N. McKenzie.
\newblock {\em Commutator Theory for Congruence Modular Varieties}, volume 125
  of {\em London Math. Soc. Lecture Note Ser.}
\newblock Cambridge University Press, 1987.
\newblock Available from
  \verb+http://math.hawaii.edu/~ralph/Commutator/comm.pdf+.

\bibitem{ALV:III}
Ralph~S. Freese, Ralph~N. McKenzie, George~F. McNulty, and Walter~F. Taylor.
\newblock {\em Algebras, lattices, varieties. {V}ol. {III}}, volume 269 of {\em
  Mathematical Surveys and Monographs}.
\newblock American Mathematical Society, Providence, RI, 2022.

\bibitem{Id:CCMO}
P.~Idziak.
\newblock Clones containing {M}al'tsev operations.
\newblock {\em International Journal of Algebra and Computation},
  9(2):213--226, 1999.

\bibitem{IMMVW:TL}
P.~Idziak, P.~Markovi\'{c}, R.~McKenzie, M.~Valeriote, and R.~Willard.
\newblock Tractability and learnability arising from algebras with few
  subpowers.
\newblock {\em SIAM J. Comput.}, 39(7):3023--3037, 2010.

\bibitem{Kom:SMP}
Michael Kompatscher.
\newblock The subpower membership problem of 2-nilpotent algebras.
\newblock In {\em 41st {I}nternational {S}ymposium on {T}heoretical {A}spects
  of {C}omputer {S}cience}, volume 289 of {\em LIPIcs. Leibniz Int. Proc.
  Inform.}, pages Paper No. 46, 17. Schloss Dagstuhl. Leibniz-Zent. Inform.,
  Wadern, 2024.

\bibitem{Ko:EXP}
Marcin Kozik.
\newblock A finite set of functions with an {EXPTIME}-complete composition
  problem.
\newblock {\em Theoret. Comput. Sci.}, 407(1-3):330--341, 2008.

\bibitem{Mal:GTAS}
A.~I. Mal'cev.
\newblock On the general theory of algebraic systems.
\newblock {\em Mat. Sb. (N.S.)}, 35(77):3--20, 1954.

\bibitem{Ma:SMP}
Peter Mayr.
\newblock The subpower membership problem for {M}al'cev algebras.
\newblock {\em Internat. J. Algebra Comput.}, 22(7):1250075, 23, 2012.

\bibitem{Ma:VLL}
Peter Mayr.
\newblock Vaughan-{L}ee's nilpotent loop of size 12 is finitely based.
\newblock {\em Algebra Universalis}, 85(1):Paper No. 2, 12, 2024.

\bibitem{MaWy:CBM}
Peter Mayr and Patrick Wynne.
\newblock Clonoids between modules.
\newblock {\em Internat. J. Algebra Comput.}, 34(4):543--570, 2024.

\bibitem{Pi:GTMF}
N.~Pippenger.
\newblock Galois theory for minors of finite functions.
\newblock {\em Discrete Math.}, 254(1-3):405--419, 2002.

\bibitem{Si:CMPG}
Charles~C. Sims.
\newblock Computational methods in the study of permutation groups.
\newblock In {\em Computational {P}roblems in {A}bstract {A}lgebra ({P}roc.
  {C}onf., {O}xford, 1967)}, pages 169--183. Pergamon, Oxford-New York-Toronto,
  Ont., 1970.

\bibitem{Va:NPV}
M.~R. Vaughan-Lee.
\newblock Nilpotence in permutable varieties.
\newblock In {\em Universal algebra and lattice theory ({P}uebla, 1982)},
  volume 1004 of {\em Lecture Notes in Math.}, pages 293--308. Springer,
  Berlin, 1983.

\bibitem{Wi:SMP}
Ross Willard.
\newblock Four unsolved problems in congruence permutable varieties.
\newblock Talk at the Conference on Order, Algebra, and Logics, 2007.
\newblock Available at
  \url{https://www.math.uwaterloo.ca/~rdwillar/documents/Slides/willard_nashville_2007_slides.pdf}.

\bibitem{Wy:CNMA}
Patrick Wynne.
\newblock {\em Clonoids and {N}ilpotent {M}al'cev {A}lgebras}.
\newblock ProQuest LLC, Ann Arbor, MI, 2024.
\newblock Thesis (Ph.D.)--University of Colorado Boulder.

\end{thebibliography}

% Bibliography added 8/19/2026 %

%
%
%

\end{document}